\documentclass[a4paper,12pt]{amsart}
\usepackage{times} 
\usepackage{latexsym,amscd,amssymb,url}
\usepackage[T1]{fontenc}         
\usepackage{ae}									
\usepackage[all,cmtip]{xy}  
\usepackage{graphicx}
\usepackage{color}
\usepackage{hyperref}
\usepackage{enumerate}

\newtheorem{theorem}{Theorem}

\theoremstyle{plain}

\newtheorem{question}[theorem]{Question}

\newtheorem{corollary}[theorem]{Corollary}
\newtheorem{definition}[theorem]{Definition}
\newtheorem{example}[theorem]{Example}

\newtheorem{lemma}[theorem]{Lemma}

\newtheorem{proposition}[theorem]{Proposition}
\newtheorem{remark}[theorem]{Remark}

\newcommand{\Z}{\mathbb Z}
\newcommand{\Q}{\mathbb Q}
\newcommand{\PP}{\mathbb P}

\newcommand{\C}{\mathbb C}
\newcommand{\N}{\mathbb N}

\newcommand{\CP}{\mathbb P}

\newcommand{\Pic}{\operatorname{Pic}}

\newcommand{\NS}{\operatorname{NS}}

\newcommand{\dashedlongrightarrow}{\xymatrix@1@=15pt{\ar@{-->}[r]&}}
\renewcommand{\longrightarrow}{\xymatrix@1@=15pt{\ar[r]&}}
\renewcommand{\mapsto}{\xymatrix@1@=15pt{\ar@{|->}[r]&}}
\renewcommand{\twoheadrightarrow}{\xymatrix@1@=15pt{\ar@{->>}[r]&}}
\newcommand{\hooklongrightarrow}{\xymatrix@1@=15pt{\ar@{^(->}[r]&}}
\newcommand{\congpf}{\xymatrix@1@=15pt{\ar[r]^-\sim&}}
\renewcommand{\cong}{\simeq}

\begin{document} 
\title[K\"ahler structures on spin $6$-manifolds]{K\"ahler structures on spin $6$-manifolds} 

\author{Stefan Schreieder}
\address{Mathematical Institute of the University of Bonn, Endenicher Allee 60, D-53115 Bonn, Germany.} 
\email{schreied@math.uni-bonn.de}

\author{Luca Tasin} 
\address{Universit\`a Roma Tre, Dipartimento di Matematica e Fisica, Largo San Leonardo Murialdo I-00146 Roma, Italy} 
\email{ltasin@mat.uniroma3.it}

\date{September 8, 2017; \copyright{\ Stefan Schreieder and Luca Tasin 2017}}



\begin{abstract}    
We show that many spin $6$-manifolds have the homotopy type but not the homeomorphism type of a K\"ahler manifold. 
Moreover, for given Betti numbers, there are only finitely many deformation types and hence topological types of  smooth complex projectve spin threefolds of general type.  
Finally, on a fixed spin $6$-manifold, the Chern numbers take on only finitely many values on all possible K\"ahler structures. 
\end{abstract}

\maketitle
\section{Introduction}

A classical question in complex algebraic geometry asks which smooth manifolds carry a complex projective or a K\"ahler structure.
Once we know that at least one such structure exists, the next natural question is to ask how many there are.
In this paper we address both questions in the case where the underlying smooth manifold is a spin $6$-manifold.
The condition for a $6$-manifold $M$ to be spin depends only on its homotopy type; if $M$ carries a complex structure $X$, then $M$ is spin if and only if the mod $2$ reduction of $c_1(X)$ vanishes.   

The spin condition, though of a purely topological nature, turns out to put strong restrictions on the bimeromorphic geometry of K\"ahler manifolds.
This makes that class of manifolds particularly accessible to the methods of the minimal model program, established for K\"ahler threefolds only very recently by H\"oring and Peternell  \cite{HP15,HP16}.
We will see that this leads to surprisingly strong topological constraints on this class of K\"ahler manifolds; our main results are Theorems \ref{thm:b_3=0}, \ref{thm:deformations} and \ref{thm:w2=0} below. 

\subsection{Many spin $6$-manifolds are non-K\"ahler but have K\"ahler homotopy type} 
Any closed spin $6$-manifold carries an almost complex structure, see Section \ref{subsubsec:almostcxstr}.
Moreover, a conjecture of Yau predicts that actually any closed spin $6$-manifold admits a complex structure, see \cite[p.\ 6]{kotschick-survey} and \cite[Problem 52]{yau-problem}.
On the other hand, our first main result produces many examples that do not carry any K\"ahler structure.
The interesting feature of our result is that it holds for a large class of prescribed homotopy types.

\begin{theorem} \label{thm:b_3=0}
Let $X$ be a simply connected K\"ahler threefold with spin structure and $H^3(X,\Z)=0$. 
Then there are infinitely many pairwise non-homeomorphic closed spin $6$-manifolds $M_i$ with the same oriented homotopy type as $X$, but which 
are not homeomorphic to a K\"ahler manifold.
\end{theorem}

To the best knowledge of the authors, the above theorem produces the first examples of manifolds with $b_2>1$ that are homotopy equivalent but not homeomorphic to K\"ahler manifolds. 
This is interesting because most known topological constraints implied by the K\"ahler condition (e.g.\ formality \cite{deligne-etal}, restrictions on the fundamental group \cite{amoros-etal}, or constraints on the cohomology algebra that come from Hodge theory)  
are only sensitive to the homotopy type, but not to the homeomorphism type.  

Simple examples to which Theorem \ref{thm:b_3=0} applies are the blow-up of $\CP^3$ in finitely many points, the product of a K3 surface with $\CP^1$, 
or, more generally, the $\CP^1$-bundle $\CP(K_S\oplus \mathcal O_S)$ over any simply connected K\"ahler surface $S$; 
the surface $S$ may very well be non-spin. 

Previously, only rather specific examples of non-K\"ahler manifolds with K\"ahler homotopy type had been found. 
For example, Libgober and Wood \cite{LW90} proved that the homotopy type of $\CP^n$ with $n\leq 6$ determines uniquely the K\"ahler structure.
This proves Theorem \ref{thm:b_3=0} in the special case $X=\CP^3$, because it is well-known that there are infinitely many pairwise non-homeomorphic $6$-manifolds that are homotopy equivalent to $\CP^3$. 
A related result of Hirzebruch--Kodaira \cite{HK57} and Yau \cite{yau} asserts that for any $n$, the homeomorphism type of $\CP^n$ determines the K\"ahler structure uniquely.

\subsection{Almost no spin $6$-manifold is a complex projective variety of general type} \label{subsec:bounded} 
In analogy with the classification of curves or surfaces, one expects that most K\"ahler structures are of general type.
Any K\"ahler structure of general type is projective and it is natural to ask which smooth manifolds carry such a structure. 

Our next result is the following strong finiteness statement. 

\begin{theorem}\label{thm:deformations}
There are only finitely many deformation types of smooth complex projective spin threefolds of general type and with bounded Betti numbers. 
\end{theorem}

The above theorem shows that the deformation type and hence the topological type of a smooth complex projective spin threefold $X$ of general type is determined up to finite ambiguity by its Betti numbers.
In particular, once we fix the Betti numbers, only finitely many fundamental groups $\pi_1(X)$, and only finitely many ring structures and torsion subgroups of $H^{\ast}(X,\Z)$ occur.
In fact, Theorem \ref{thm:deformations} implies the following.

\begin{corollary} \label{cor:deformations}
Only finitely many closed spin $6$-manifolds with bounded Betti numbers carry 
a K\"ahler structure of general type.
\end{corollary}

Corollary \ref{cor:deformations} has the surprising consequence that almost no closed spin $6$-manifold is a complex projective variety of general type.
To see this, we start with the case of simply connected spin $6$-manifolds $M$, which are completely classified by the work of Wall, see Section \ref{subsubsec:classification}.
It turns out that for given Betti numbers with $b_2>0$, the ring structure on $H^*(M,\Z)$ may vary among infinitely many different isomorphism types and almost none of these manifolds can be a variety of general type. 
In fact, even the corresponding homotopy types cannot be realized by a variety of general type; non-simply connected examples can be produced similarly by taking connected sums with simply connected ones.
Going further, it is also possible to produce many new examples of manifolds that are homotopy equivalent but not homeomorphic to a variety of general type, see Corollary \ref{cor:6mfd-not-gentyp} below.

Although it is well--known that Theorem \ref{thm:deformations} holds in dimension two, this cannot be exploited in a similar way.
The problem being that the oriented homeomorphism type of a simply connected smooth $4$-manifold $M$ is determined by the cup product pairing on $H^2(M,\Z)$ by Freedman's work, and only very few (in particular finitely many) isomorphism types occur by Donaldson's theorem. 

Our arguments are new and do not seem to apply to other natural classes of K\"ahler manifolds.
Indeed, it is easy to see that Theorem \ref{thm:deformations} and Corollary \ref{cor:deformations} fail in higher dimensions.
Moreover, both statements are sharp in the sense that they fail if one drops either the spin assumption, or the general type assumption, see Example \ref{ex:deformations:gentype} and Proposition \ref{prop:dolgachev} below. 
Nonetheless, Theorem \ref{thm:b_3=0} and \ref{thm:deformations} lead to the expectation that in some sense, most spin $6$-manifolds should not carry any K\"ahler structure at all. 

Finiteness results like Theorem \ref{thm:deformations} seem to be quite rare; in fact, apart from the aforementioned case of surfaces, we are aware of only one more result in this direction. 
That result is due to Koll\'ar, who showed \cite[Theorem 4.2.3]{kollar} that all K\"ahler manifolds with $b_2=1$, given cohomology ring and Pontryagin classes form a bounded family.
Minimal ruled surfaces show that this statement fails already for $b_2=2$.
Moreover, for applications like Theorem \ref{thm:b_3=0} or Corollary \ref{cor:deformations}, it is crucial not to fix the Pontryagin classes.

\subsection{K\"ahler structures on spin $6$-manifolds have bounded Chern numbers} 
Let us now assume that a given spin $6$-manifold $M$ carries a K\"ahler structure and we ask how many there are.
More precisely, it is natural to ask which (homotopy classes of) almost complex structures can be realized by a K\"ahler structure.
Since $M$ is spin, the first Chern class induces a bijection between the homotopy classes of almost complex structures on $M$ and the subgroup $H^2(M,2\Z)$ of classes in $H^2(M,\Z)$ whose mod $2$ reduction vanishes, see Proposition \ref{prop:almostcx}. 
It follows from Theorem \ref{thm:deformations} that only finitely many such classes come from K\"ahler structures of general type.
This fails without the general type assumption: there are spin $6$-manifolds that admit infinitely many K\"ahler structures with unbounded first Chern class, see Proposition \ref{prop:dolgachev}. 
Nonetheless, our next result shows that the Chern numbers are always bounded. 

\begin{theorem} \label{thm:w2=0}
Let $X$ be a smooth compact K\"ahler threefold with spin structure. 
Then the Chern numbers of $X$ are determined up to finite ambiguity by the isomorphism type of the cohomology ring $H^{\ast}(X,\Z)$ and the first Pontryagin class $p_1(X)$.  
\end{theorem} 
 
The following corollary of the above theorem solves a problem of Kotschick in the case of spin structures \cite[Problem 1]{kotschick-JTOP}.
 
\begin{corollary} \label{cor:Spin}
The Chern numbers take on only finitely many values on the K\"ahler structures with the same underlying closed spin $6$-manifold.  
\end{corollary}

While the boundedness of $c_3$ and $c_1c_2$ on K\"ahler structures with the same underlying smooth $6$-manifold is well-known, the significance of the above results concerns the boundedness of $c_1^3$.
For non-uniruled K\"ahler threefolds, our bounds for $c_1^3$ are effective and depend only on the Betti numbers, see Corollary \ref{cor:ineq}.

The original motivation for Kotschick's question goes back to a problem of Hirzebruch, asking which linear combinations of Chern and Hodge numbers of a smooth complex projective variety are determined by the underlying smooth manifold \cite{kotschick-JTOP,kotschick-PNAS,kotschick-advances,kotschick-schreieder}.

Interestingly, Theorem \ref{thm:w2=0} and Corollary \ref{cor:Spin} fail if one drops the K\"ahler assumption  \cite{lebrun}. %
Moreover, none of the statements generalizes to higher dimensions. 
Indeed, for any $n\geq 4$ there are smooth $2n$-manifolds with infinitely many K\"ahler structures whose Chern numbers are unbounded, see \cite{schreieder-tasin}. 
The method of that paper can be adapted to produce similar examples with spin structures; more precisely, \cite[Theorem 1]{schreieder-tasin} remains true for spin manifolds.

\begin{remark} \label{rem:diffeotypes}
Any topological $6$-manifold admits at most one equivalence class of smooth structures; the existence of such a structure is detected by the Kirby--Siebenmann invariant. 
In particular, Corollary \ref{cor:Spin} remains true if one replaces smooth spin $6$-manifolds by topological ones.
\end{remark}

\subsection{Conventions}  
All manifolds are smooth, closed and connected.
A K\"ahler manifold is a complex manifold which admits a K\"ahler metric. 
All schemes are separated; a variety is an integral scheme of finite type over $\C$.
A family of varieties is a proper flat morphism $\pi:\mathcal X\longrightarrow B$ between finite type schemes over $\C$ whose fibres over closed points are varieties; in this paper, the base $B$ will always be assumed to be reduced. 
For any (K\"ahler) manifold $X$, we denote by $H^\ast_{tf}(X,\Z)$ the quotient $H^\ast(X,\Z)/H^\ast(X,\Z)_{tors}$, where $ H^\ast(X,\Z)_{tors}$ is the torsion subgroup of $H^\ast(X,\Z)$. 

\section{Preliminaries} \label{sec:preliminaries}

\subsection{Spin $6$-manifolds}
A smooth orientable $6$-manifold $M$ is spin if and only if the second Stiefel--Whitney class $w_2(M)\in H^2(M,\Z\slash 2\Z)$ vanishes.
It follows from the Wu formula that this condition depends only on the homotopy type of $M$.
Moreover, if $X$ is a complex structure on $M$, then the mod $2$ reduction of $c_1(X)$ coincides with $w_2(M)$ and so it vanishes if and only if $M$ is spin.

For example, by the adjunction formula, a smooth divisor $X$ on a compact complex fourfold $Y$ is spin if it is linearly equivalent to $K_Y+2L$ for some $L\in \Pic(Y)$. 
In particular, any smooth complex projective fourfold contains (many) smooth hypersurfaces $X$ with spin structure.
The smooth $6$-manifold which underlies $X$, or the blow-up of $X$ in a finite number of points, is spin and so the results of this paper apply.

The following definition generalises the notion of spin threefold; note however that this generalization is not a topological notion anymore.  

\begin{definition}\label{def:divisible}
Let $X$ be a smooth complex projective threefold. 
We say that $c_1(X)$ is \emph{numerically divisible by $m \in \N$} if its class in $H^2_{tf}(X,\Z)$ is divisible by $m$.
Moreover, $c_1(X)$ is \emph{numerically divisible} if it is divisible by some natural number $m\geq 2$.
\end{definition}

\subsubsection{Almost complex structures.} \label{subsubsec:almostcxstr}
The obstruction for a closed oriented $6$-manifold $M$ to carry an almost complex structure is given by the image of the second Stiefel--Whitney class via the Bockstein homomorphism $H^2(M,\Z\slash 2\Z)\longrightarrow H^3(M,\Z)$ and hence it vanishes if $M$ is spin. 
In fact, we have the following,
see proof of \cite[Proposition 8]{OV95}. 

\begin{proposition} \label{prop:almostcx}
Let $M$ be a closed spin $6$-manifold. 
Then the first Chern class induces a bijection between the homotopy classes of almost complex structures on $M$ and the subgroup $H^2(M,2\Z)\subset H^2(M,\Z)$ of classes whose mod $2$ reduction vanishes.  
\end{proposition} 

\subsubsection{Classification of simply connected spin $6$-manifolds} \label{subsubsec:classification}
Wall \cite{wall} showed that simply connected closed spin $6$-manifolds $M$ with torsion free cohomology are classified by the following tuple  of algebraic invariants
$$
(b_3(M),H^2(M,\Z),F_M,p_1(M)) ,
$$
where $F_M$ denotes the cubic form on $H^2(M,\Z)$, given by cup product, and $p_1(M)$ is the linear form on $H^2(M,\Z)$, given by cup product with the first Pontryagin class, see also \cite[Section 1]{OV95}. 
Two such manifolds $M$ and $M'$ are orientation-preservingly homeomorphic (or diffeomorphic, see Remark \ref{rem:diffeotypes} above) if and only if there is an isomorphism between $H^2(M,\Z)$ and $H^2(M',\Z)$ which respects the above tuple of algebraic invariants.
A given tuple 
$
(b_3,H,F,p_1)
$
can be realized by a simply-connected smooth spin $6$-manifold with torsion-free cohomology if and only if $b_3$ is even and
\begin{align} \label{eq:admissible}
4W^3\equiv p_1(M)W \mod 24
\end{align}
for all $W\in H$; such tuples are called admissible.

\subsection{Chern numbers of smooth threefolds}
Let $X$ be a smooth compact complex threefold.  
Since $c_3(X)=\sum_i(-1)^ib_i(X)$ coincides with the topological Euler number, it is a topological invariant of $X$.
Moreover, by Riemann--Roch, $c_1c_2(X)=24\chi(X,\mathcal O_X)$ is bounded by the Hodge numbers, hence by the Betti numbers if $X$ is K\"ahler.  
In particular, in order to prove Theorem \ref{thm:w2=0}, we only need to prove the boundedness of $c_1^3$. 
This is known if $K_X$ is nef. 

\begin{proposition} \label{prop:gentype}
Let $X$ be a smooth K\"ahler threefold such that $K_X$ is nef.
Then $c_1^3(X)$ is bounded by the Betti numbers of $X$.
\end{proposition}

\begin{proof} 
If $X$ is of general type, then it is projective because it is Moishezon and K\"ahler.
Following an observation of Kotschick \cite{kotschick-JTOP}, the boundedness of $c_1^3(X)$ follows then from the Miyaoka--Yau inequality 
\begin{align} \label{eq:Yauineq}
0>c_1^3(X)\geq \frac{8}{3}c_1c_2(X), 
\end{align}
see \cite{tsuji,zhang}.
If $X$ has Kodaira dimension $0,1$ or $2$, then $c_1^3(X)=0$, cf.\ \cite{CHP16}.   
\end{proof}

\section{The bimeromorphic geometry of K\"ahler threefolds with spin structure} \label{sec:MMP}

One of the key ideas of this paper is the observation that for a K\"ahler threefold, the purely topological condition of being spin puts strong restrictions on its bimeromorphic geometry.
This is a consequence of the minimal model program for K\"ahler threefolds \cite{HP15,HP16}, together with some classical results of Mori \cite{mori-annals}. 

Before we state the result, let us recall that a K\"ahler manifold $Y$ is a Mori fibre space, if there is a proper morphism $f:Y\longrightarrow B$ with positive dimensional connected fibres onto a normal $\Q$-factorial K\"ahler space $B$ with at most klt singularities, such that $-K_Y$ is $f$-ample and the relative Picard number is one, see \cite{HP15}.
In this paper, we will only need to deal with the special case where $B$ is smooth. 

\begin{theorem}[\cite{HP15,HP16,mori-annals}]  \label{thm:MMP} 
Let $X$ be a smooth K\"ahler threefold such that $c_1(X)$ is numerically divisible by $2$ (e.g.\ a spin threefold). 
Then there is a finite sequence of blow-downs to smooth points
$$
X=Y_r\longrightarrow Y_{r-1}\longrightarrow \dots \longrightarrow Y_1\longrightarrow Y_0=Y,
$$ 
such that $Y$ is a smooth K\"ahler threefold with $c_1(Y)$ numerically divisible by $2$, which is either a minimal model (i.e.\ $K_Y$ is nef) or a Mori fibre space.  
Moreover, if $Y$ is a Mori fibre space, then it is one of the following: 
\begin{enumerate}  
\item an unramified conic bundle over a smooth K\"ahler surface; 
\item a quadric bundle over a smooth curve; 
\item a smooth Fano threefold. 
\end{enumerate}
\end{theorem}

\begin{proof} 
By \cite{HP15,HP16}, we can run a minimal model program on $X$.
If $X$ is not a minimal model, nor a Mori fibre space, then, since there are no flips in the smooth category \cite{mori-annals, mori-flips}, there is a divisorial contraction $f:X\longrightarrow Y$ with exceptional divisor $E$.
Since $c_1(X)$ is numerically divisible by 2, we get that $K_X.C\leq -2$ for every contracted curve $C\subset E$.
It is known to experts, that the theorem follows now from \cite{mori-annals} and \cite{mori-prokhorov}; we sketch the proof for convenience of the reader.

Since $-K_X$ is $f$-ample, Mori's classification of all possible exceptional divisors applies \cite[Theorem 3.3]{mori-annals}.
In particular, $f$ contracts $E$ either to a smooth curve, or to a point.
In the former case, $f$ is the blow-up of a smooth curve of $Y$ and so the general fibre of $E\longrightarrow f(E)$ is a rational curve $C\subset E$ with $K_X.C=-1$, which contradicts our assumptions. 
In the latter case, $E$ together with its normal bundle in $X$ is isomorphic to $(\CP^2,\mathcal O(1))$, $(\CP^2,\mathcal O(2))$ or $(Q,\mathcal O(1))$, where $Q\subset \CP^3$ is an integral quadric.
If $(E,N_{E\slash X})=(Q,\mathcal O(1))$, then $K_X=f^{\ast}K_Y+E$ which implies $K_X.\ell=-1$ for a line $\ell\subset Q$.
If $(E,N_{E\slash X})=(\CP^2,\mathcal O(2))$, then $K_X=f^{\ast}K_Y+\frac{1}{2}E$ which implies again $K_X.\ell=-1$ for a line $\ell\subset \CP^2$.
Since $K_X$ is numerically divisible by 2, none of the two possibilities occur, which shows $(E,N_{E\slash X})=(\CP^2,\mathcal O(1))$, and so $f$ is the blow-down to a smooth point of $Y$. 
In particular, $K_X=f^{\ast}K_Y+2E$ and so $Y$ is a smooth K\"ahler threefold  such that $c_1(Y)$ is numerically divisible by 2 and $b_2(Y)=b_2(X)-1$. 
After a finite number of such blow-downs, we may assume that $Y$ is either a minimal model, or a Mori fibre space.

It remains to deal with the case where $f:Y\longrightarrow B$ is a Mori fibre space.
If $\dim(B)\leq 1$, then $h^{2,0}(Y)=0$ and so $Y$ is projective by Kodaira's criterion.  
It thus follows from \cite[Theorem 3.5]{mori-annals} that $B$ is smooth.
Moreover, if $\dim(B)=1$, then the general fibre of $f$ is a smooth del Pezzo surface with spin structure, hence a smooth quadric in $\CP^3$.
This concludes the case $\dim(B)\leq 1$.

If $\dim(B)=2$, then, by \cite[Corollary 2.4.2]{mori-prokhorov}, $B$ is smooth and there is a rank $3$ vector bundle $\mathcal E$ on $B$ together with an inclusion $Y\subset \CP(\mathcal E)$ which realises each fiber $f^{-1}(b)$ as a conic in $\CP(\mathcal E_b)$.
If $f^{-1}(b)$ is singular, then it contains a line $\ell$ with $K_{Y}.\ell=-1$, which contradicts the assumption that $c_1(Y)$ is numerically divisible.
This proves that $f$ is an unramified conic bundle over a smooth K\"ahler surface, which finishes the proof.
\end{proof}

The same argument as above, together with the classification of smooth Fano threefolds \cite[\S 12]{iskovski}, shows the following.

\begin{theorem}[\cite{HP15,HP16,mori-annals}] \label{thm:p>2}
Let $X$ be a smooth K\"ahler threefold such that $c_1(X)$ is numerically divisible by some integer $m \ge 3$.
Then one of the following 
holds:
\begin{enumerate}
\item $X$ is a minimal model, i.e.\ $K_X$ is nef;  
\item $m=3$ and $X$ is a $\CP^2$-bundle over a smooth curve;  
\item $m=3$ and $X$ is isomorphic to a smooth quadric in $\CP^4$;
\item $m=4$ and $X$ is isomorphic to $\CP^3$.
\end{enumerate} 
\end{theorem}
%

\section{Non-K\"ahler manifolds with K\"ahler homotopy types} \label{sec:b_3=0}

In this section we produce many non-K\"ahler manifolds with the homotopy type of K\"ahler manifolds.
The main results are Theorem \ref{thm:b_3=0:2} 
and Proposition \ref{prop:b_2=1} below.
We start with the following lemma. 

\begin{lemma} \label{lem:p1(Mfs)} 
Let $X$ be a smooth K\"ahler threefold admitting the structure of an unramified conic bundle $f:X\longrightarrow S$ over a smooth K\"ahler surface $S$.
Then, $p_1(X)\in f^{\ast}H^4(S,\Z)$.
\end{lemma}

\begin{proof}
If $X$ is projective, then the first statement follows from the observation that for any smooth curve $C\subset S$, the preimage $R:=f^{-1}(C)$ is a ruled surface with normal bundle $f^{\ast}\mathcal O_S(C)|_R$ and so $p_1(X).R=p_1(R)+p_1(f^{\ast}\mathcal O_S(C)|_{R})=0$.
The second statement can be proven similarly.

In general, the lemma follows from the following topological argument, suggested to us by M.\ Land respectively.
By \cite{smale}, any smooth $S^2$-bundle $f:M\longrightarrow S$ can be realized as the sphere bundle of an oriented real rank three vector bundle $E$ on $S$.
Since the tangent bundle of a sphere is stably trivial, $TM$ is stably isomorphic to $f^{\ast}(TS\oplus E)$ and so $p_1(M)=f^*p_1(TS\oplus E)$. 
%
%
\end{proof}

The next result implies Theorem \ref{thm:b_3=0} stated in the introduction.

\begin{theorem} \label{thm:b_3=0:2}
Let $X$ be a simply connected spin 3-fold with $b_2(X)\geq 1$ and $H^3(X,\Z)=0$. 
Then there are infinitely many pairwise non-homeomorphic spin $6$-manifolds $M_i$ with the oriented homotopy type of $X$, and which are not homeomorphic to a K\"ahler manifold.
\end{theorem}

\begin{proof} 
Since $X$ is simply connected with $H^3(X,\Z)=0$, $X$ has torsion free cohomology, and so the classification of Wall applies, see Section \ref{subsubsec:classification}.
Let 
$$
(0,H^2(X,\Z),F_X,p_1(X))
$$ 
be the admissible tuple which corresponds to $X$. 

In order to construct examples that are homotopy equivalent to $X$, we fix a general element
$$
\omega\in H^2(X,\Z)^{\vee} \ \ \text{with}\ \ \omega \equiv 0 \mod 48 .
$$  
General means here that its image in rational cohomology lies outside a finite number of proper subvarieties that we will encounter in the process of the proof; it is important to note that these subvarieties are going to depend only on the ring structure of $H^*(X,\Q)$.

For any integer $r\equiv 1 \mod 48$, the tuple
$$
(0,H^2(X,\Z),F_X,r(p_1(X)+\omega))
$$ 
is admissible, see (\ref{eq:admissible}).
Let $M_{r}$ be the corresponding spin $6$-manifold. 
By work of Zhubr (see \cite{zhubr} or \cite[Theorem 2]{OV95}), the homotopy type of $M_{r}$ depends only on the reduction modulo $48$ of $p_1(M_r)$.
Our choices ensure therefore that $M_{r}$ is homotopy equivalent to $X$.  
Since the image of the first Pontryagin class in $H^4_{t.f.}(X,\Z)$ ($=H^4(X,\Z)$) is a homeomorphism invariant by Novikov's theorem \cite{novikov}, $M_{r}$ and $M_{r'}$ are not homeomorphic if $r\neq r'$, see also Remark \ref{rem:diffeotypes}.
In order to conclude Theorem \ref{thm:b_3=0:2}, it thus suffices to show that for infinitely many values of $r$, $M_{r}$ is not homeomorphic to a K\"ahler manifold. 

For a contradiction, we assume from now on that $r>>0$ is sufficiently large and that there is a K\"ahler manifold $X_{r}$ which is homeomorphic to $M_{r}$. 
By construction, there is a natural isomorphism
\begin{align}\label{H*X}
H^*(X_{r},\Z)\cong H^*(X,\Z).
\end{align} 
By Novikov's theorem, or Remark \ref{rem:diffeotypes}, the above isomorphism identifies $p_1(X_r)$ with $p_1(M_r)=r(p_1(X)+\omega)$. 
Since the first Pontryagin class of a blow-up of a complex threefold in a point is not divisible by any integer greater than $4$, see \cite[Proposition 13]{OV95}, it follows from Theorem \ref{thm:MMP} that $X_{r}$ is minimal or a Mori fibre space.

Since $X_{r}$ is K\"ahler with $b_1(X_{r})=b_3(X_{r})=0$, Riemann--Roch says
\begin{align}\label{eq:RR}
c_1c_2(X_{r})=24+24h^{2,0}(X_{r}) >0.
\end{align}
The Miyaoka--Yau inequality (\ref{eq:Yauineq}) shows then that $X_{r}$ is not of general type.
We claim that it must in fact be of negative Kodaira dimension.
Indeed, otherwise $c_1^3(X_{r})=0$ and so
$$
p_1(X_{r})c_1(X_{r})=c_1^3(X_{r})-2c_1c_2(X_{r})=-48-48 h^{2,0}(X_{r})
$$
would be non-zero and very divisible, whereas $h^{2,0}(X_{r})$ is bounded by $b_2(X_{r})=b_2(X)$.

We have thus shown that $X_{r}$ is a Mori fibre space.
Since there are only finitely many deformation types of Fano threefolds, and since $p_1(X_{r})$ is sufficiently divisible, $X_{r}$ cannot be Fano.

Let us assume that $f:X_{r}\longrightarrow C$ is a Mori fibre space over a curve. 
A  general fibre $Q$ of $f$ is a smooth quadric surface and so 
$p_1(X_{r}).Q=p_1(Q)=0 
$. 
On the other hand, the line $[Q]\cdot \Q\subset H^2(X_{r},\Q)$ lies in the locus where the cubic form vanishes and so it is, up to at most three possibilities, uniquely determined by (\ref{H*X}).
(This uses $b_2(X_{r})=2$, which follows from $\rho(X_r\slash C)=1$ and $h^{2,0}(X_r)=0$.)
Since $\omega$ is sufficiently general, the restriction of $p_1(X_{r})\in H^2(X_{r},\Q)^{\vee}$ to any of these lines will be nonzero and so this case cannot happen. 

It remains to deal with the case where $f:X_{r}\longrightarrow S$ is a Mori fibre space over a K\"ahler surface $S$.
By Theorem \ref{thm:MMP}, $f$ is an unramified conic bundle.
Since $f^{\ast}H^2(S,\Z)\subset H^2(X_{r},\Z)$ is a hyperplane which is contained in the vanishing locus of the cubic form, it follows from (\ref{H*X}) that the subspace $f^{\ast}H^2(S,\Z)\subset H^2(X_{r},\Z)$ is determined uniquely up to at most three possibilities.
Therefore, the line in $H^4(X_r,\Q)$ that is spanned by the class of a fibre of $f$ is also determined up to at most three possibilities.
Since $\omega$ is general, $p_1(X_{r})$ is not contained in any of those lines.
This contradicts Lemma \ref{lem:p1(Mfs)}, which finishes the proof of the theorem.  
\end{proof}

\begin{remark} 
Any simply connected spin $6$-manifold $M$ with torsion free cohomology can be written as a connected sum $M=M_0\sharp^{b_3/2} (S^3\times S^3)$, with $H^3(M_0,\Z)=0$;
the diffeomorphism type of $M_0$, called the core of $M$, is uniquely determined by $M$, cf.\ \cite[Corollary 1]{OV95}.
In particular, Theorem \ref{thm:b_3=0:2} applies to the core $M_0$ of any such manifold $M$ as long as $b_2(M)>0$.
\end{remark}

\begin{remark} \label{rem:b_3=0:nongeneraltype}
A K\"ahler threefold $X$ with spin structure and $b_3(X)=0$ cannot be of general type.
Indeed, if it was of general type, then we may by Theorem \ref{thm:MMP} assume that it is minimal and so the Miyaoka--Yau inequality (\ref{eq:Yauineq}) holds. 
In particular, $c_1c_2(X)$ is negative, which contradicts Riemann--Roch (\ref{eq:RR}), because $b_1(X)=0$ by Hard Lefschetz. 
\end{remark}

In view of  Theorem \ref{thm:b_3=0} and Remark \ref{rem:b_3=0:nongeneraltype}, it is natural to ask for non-K\"ahler manifolds with the homotopy type of a variety of general type. 
We were unable to find such examples in the literature. 
The next result fills this gap; it applies for instance to arbitrary complete intersection threefolds $X\subset \CP^N$, which might very well be non-spin.

\begin{proposition} \label{prop:b_2=1}
Let $X$ be a simply connected $6$-manifold with $H_2(X,\Z)\cong\Z$.
Then there are infinitely many pairwise non-homeomorphic $6$-manifolds with the same oriented homotopy type as $X$, but which are not homeomorphic to a K\"ahler manifold.
\end{proposition}

\begin{proof} 
Let $Y$ be a K\"ahler threefold which is homotopy equivalent to $X$.
Since $b_2(Y)=1$, $Y$ is either Fano, Calabi-Yau or $K_Y$ is ample.
Moreover, the ample generator $L$ of $H^2(Y,\Z)$ is uniquely determined by the topological property $L^3>0$.
If $Y$ is Calabi-Yau, then $p_1(Y)=-2c_2(Y)$ and so
\begin{align} \label{eq:miyaoka}
p_1(Y).L<0 
\end{align} 
by Miyaoka's inequality \cite[Theorem 1.1]{miyaoka87}.

In order to construct infinitely many pairwise non-homeomorphic manifolds $M_i$ with the homotopy type of $X$, we may now proceed similar to the proof of Theorem \ref{thm:b_3=0}.
By the classification of simply connected $6$-manifolds with torsion free cohomology, see \cite[Section 1]{OV95}, we can choose $p_1(M_i)$ sufficiently divisible and of a sign which violates Miyaoka's inequality (\ref{eq:miyaoka}).  
It follows that $M_i$ cannot support a Calabi--Yau structure.
For sufficiently divisible $p_1(M_i)$, 
 the cases where the canonical bundles are ample or anti-ample cannot happen either, because the corresponding families are bounded by the boundedness of Fano threefolds and Matsusaka's big theorem together with the Miyaoka--Yau inequality, respectively, cf.\ \cite[Proposition 28]{schreieder-GT}.
\end{proof}

\begin{remark}
Borel's conjecture predicts that a homotopy equivalence between closed manifolds with contractible universal cover is homotopic to a homeomorphism; this is known for large classes of fundamental groups, such as hyperbolic groups, see \cite{bartels-lück}. 
Thus the assumption on $\pi_1(X)$ 
in Theorem \ref{thm:b_3=0:2} and Proposition \ref{prop:b_2=1}  
is important. 
\end{remark} 

By Remark \ref{rem:b_3=0:nongeneraltype}, the following question remains open, but see Corollary \ref{cor:6mfd-not-gentyp} below, where a weaker statement is proven.

\begin{question}
Is there a non-K\"ahler manifold with large $b_2$ 
that has the oriented homotopy type of a smooth complex projective variety of general type?
\end{question}


As mentioned in the introduction, most topological constraints known for K\"ahler or projective manifolds are only sensitive to the homotopy type, but not to the homeomorphism type.
For instance, Voisin used restrictions on the cohomology algebra to produce K\"ahler manifolds that are not homotopy equivalent to smooth complex projective varieties \cite{voisin-kodaira}, but the following question remained open. 

\begin{question} \label{question:voisin}
Is there a K\"ahler manifold which has the oriented homotopy type but not the homeomorphism type of a smooth complex projective variety?
\end{question}

\section{Proof of Theorem \ref{thm:deformations}} \label{sec:gentype}  

Theorem \ref{thm:deformations} stated in the introduction follows from the following result.

\begin{theorem}\label{thm:deformations:2} 
Let $b\in \N$.
Then there are only finitely many deformation types of smooth complex projective threefolds $X$ of general type, such that $c_1(X)$ is numerically divisible and $b_i(X)\leq b$ for all $i$. 
\end{theorem}

\begin{proof} 
Let $X$ be a smooth complex projective threefold of general type such that $c_1(X)$ is numerically divisible.
By Theorems \ref{thm:MMP} and \ref{thm:p>2}, there is some $r\geq 0$ and a smooth complex projective minimal threefold $Y$ of general type such that $X$ is obtained from $Y$ by a finite sequence of $r$ blow-ups along points.
Clearly, any variety that is obtained from $Y$ by a sequence of $r$ blow-ups along points is deformation equivalent to $X$.
Moreover, the number $r$ of blow-ups is bounded from above by $b_2(X)-1$.

By \cite[Theorem 1.2]{cascini-tasin} (or Proposition \ref{prop:gentype} above), there is a constant $c$ which depends only on the Betti numbers of $X$ such that 
$$
\operatorname{vol}(Y,K_Y)\leq c .
$$ 
In order to prove Theorem \ref{thm:deformations:2}, it is therefore enough to prove that there are only finitely many deformation equivalence classes of smooth complex projective minimal threefolds of general type and with volume bounded by $c$.

By \cite{HM06,takayama,tsuji2}, threefolds of  general type and with bounded volume are birationally bounded.
That is, there is a projective morphism of normal quasi-projective schemes $\pi^{bdd}: \mathcal X^{bdd} \longrightarrow B$ such that any smooth complex projective threefold of general type and of volume at most $c$ is birational to a fibre of $\pi^{bdd}$. 
By Noetherian induction, replacing $B$ by a disjoint union of locally closed subsets and resolving $\mathcal X^{bdd}$, we get a smooth projective family $\pi^{sm}: \mathcal X^{sm} \longrightarrow B$. 
By the deformation invariance of plurigenera \cite{siu}, we may assume that any fibre of $\pi^{sm}$ is of general type.
Up to replacing $B$ by a finite \'etale covering, we can apply the MMP in families \cite[Theorem 12.4.2]{KM92} to $\pi^{sm}$ and obtain a projective family $\pi: \mathcal X \longrightarrow B$ of minimal models of general type, such that any smooth complex projective threefold of general type and with volume bounded by $c$ is birational to a fibre of $\pi$. 

In the next step of the proof, we use the family $\pi$ to construct a second family $\pi':\mathcal X'\longrightarrow B'$ with the following property:
\begin{enumerate}
\item[($\ast$)] 
if $b\in B_i$ is a very general point  of some component $B_i$ of $B$, then the fibre $X_b=\pi^{-1}(b)$ has the property that any of its minimal models is isomorphic to a fibre of $\pi':\mathcal X'\longrightarrow B'$.
\end{enumerate}

In order to explain our construction, let $B_i$ be a component of $B$ and consider the geometric generic fibre $\mathcal X_{i,\overline {\eta}}$ of $\mathcal X_i\longrightarrow B_i$.
The number of minimal models of $\mathcal X_{i,\overline {\eta}}$ is finite by \cite{KM87} and they are all connected by flops by \cite{kollar89}.
If $\mathcal X_{i,\overline {\eta}}'$ is such a minimal model, then there is a sequence of flops $\mathcal X_{i,\overline {\eta}}\dashrightarrow \mathcal X_{i,\overline {\eta}}'$.
This sequence of flops is defined over some finite extension of the function field $k(B_i)$, hence over the ring of regular functions of some affine variety $B_i'$ which maps finitely onto a Zariski open and dense subset of $B_i$.
In particular, we can spread $\mathcal X_{i,\overline {\eta}}'$  over $B_i'$ to get a family 
$$
\mathcal X'_{i}\longrightarrow B_i' ,
$$ 
whose geometric generic fibre is $\mathcal X_{i,\overline {\eta}}'$.

We define $\pi':\mathcal X'\longrightarrow B'$ to be the disjoint union of all such families $\mathcal X'_{i}\longrightarrow B_i'$ that we can construct in the above way from all the minimal models of all geometric generic fibres $\mathcal X_{i,\overline \eta}$ of $\pi$.
In particular, the number of irreducible components of $B'$ corresponds to the number of minimal models of the geometric generic fibres $\mathcal X_{i,\overline \eta}$ and so $B'$ has finitely many components.
By construction, any minimal model of any $\mathcal X_{i,\overline \eta}$ appears as geometric generic fibre of $\pi'$ over some component of $B'$.
Since the geometric generic fibre is abstractly (as scheme over $\Z$) isomorphic to any very general fibre, this implies that property ($\ast$)  holds; for convenience of the reader we give some details below.
\begin{lemma}
Property ($\ast$) holds.
\end{lemma}
\begin{proof}
Let $X_b$ be the fibre above a very general point $b\in B_i$; and let $k$ be a finitely generated extension of $\Q$ over which $B_i$ can be defined.
Since $B_i'$ is affine, $B_i\subset \mathbb A^N$ for some $N$ and so we may write $b=(b_1,\dots ,b_N)$. 
We can choose an isomorphism of fields 
$$
\sigma:\C\stackrel{\sim} \longrightarrow \overline{\C(B_i)} ,
$$ 
which restricts to an isomorphism between the subfield $k(b_1,\dots ,b_N)$ and the function field of $B_i$ over $k$. 
The field automorphism $\sigma$ induces an isomorphism 
$$
\varphi: X_b \longrightarrow   \mathcal X_{i,\overline \eta} 
$$  
of schemes over $\Z$ (on stalks, this map is not $\C$-linear but $\sigma^{-1}$-linear). 
Any sequence of flops $X_b\dashrightarrow X^+_b$ corresponds via $\varphi$ to a sequence of flops of the geometric generic fibre $\mathcal X_{i, \overline \eta}\dashrightarrow \mathcal X_{i, \overline \eta}^+$. 
Moreover, $\mathcal X_{i, \overline \eta}^+$ is a minimal model of $\mathcal X_{i, \overline \eta}$ because it is obtained by a sequence of flops.
By the construction of the family $\pi'$, $\mathcal X_{i, \overline \eta}^+$ is isomorphic to the geometric generic fibre of $\pi'$ over some component $B'_j$ of $B'$:
$$
\mathcal X_{i, \overline \eta}^+\cong \mathcal X'\times_{B'}\overline {\C(B'_j)} .
$$
By construction of $\pi'$, $B'_j$ maps finitely onto a Zariski open subset of $B_i$, which yields an identification $\overline {\C(B_i)}= \overline {\C(B'_j)}$.
Therefore, $\sigma$ induces a ($\sigma$-linear) isomorphism 
$$
\psi:\mathcal X_{i, \overline \eta}^+\longrightarrow X'_{b'}
$$ 
of schemes over $\Z$, where $X'_{b'}={\pi'}^{-1}(b')$ is a fibre of $\pi'$ for some point $b'\in B_j'$ which maps to $b$ via $B_j'\longrightarrow B_i$. 
Moreover, the composition
$$
X_b \stackrel{\varphi}\longrightarrow  \mathcal X_{i,\overline \eta}\dashedlongrightarrow \mathcal X_{i,\overline \eta}^+ \stackrel{\psi}\longrightarrow  X_{b'}' ,
$$
is a rational map $X_b\dashrightarrow X'_{b'}$ which is $\C$-linear on stalks.
Therefore, $X_b\dashrightarrow X'_{b'}$ is a sequence of flops of complex projective varieties, which identifies $X_{b'}'$ with $X_b^+$. 
This proves that ($\ast$) holds.
\end{proof}

Let us now consider the following set of deformation equivalence classes of smooth fibres $X'_{b'}:={\pi'}^{-1}(b')$ of $\pi'$:
$$
\mathcal S:=\{[X'_{b'}] \mid  b'\in B'\ \text{such that $X'_{b'}$ is smooth}\} .
$$ 
Since $B'$ has finitely many irreducible components, and since smoothness is an open condition,  $\mathcal S$ is finite. 

To conclude, let $X$ be any smooth complex projective minimal threefold of general type and with $\operatorname{vol}(X,K_X)\leq c$. 
We claim that the deformation equivalence class $[X]$ belongs to $\mathcal S$, which implies the theorem. 
To prove the claim, recall that $\pi:\mathcal X\longrightarrow B$ is a family of complex projective minimal threefolds with the property that any smooth threefold of general type whose volume is bounded by $c$ is birational to a fibre of $\pi$. 
In particular, there is a component $B_{i_0}$ and a point $0\in B_{i_0}$ such that the fibre $X_0$ above $0$ is birational to $X$.
Then, since $X$ and $X_0$ are birational minimal models, they are connected by a sequence of flops.
Therefore, by \cite[Theorem 11.10]{KM92}, we can find an analytic neighbourhood $U\subset B_{i_0}$ of $0 \in B_{i_0}$, such that the base change $\mathcal X_U \longrightarrow U$ admits a sequence of flops to get a family $\mathcal X^+_U \longrightarrow U$ whose central fibre is isomorphic to $X$. 
Since $X$ is projective, we may by \cite[Theorem 12.2.10]{KM92} assume that all fibres of $\mathcal X^+_U \longrightarrow U$ are projective.
Moreover, a very general fibre $X^+_t$ of $\mathcal X^+_U \longrightarrow U$ is connected to a very general fibre of $\mathcal X_U \longrightarrow U$ by a sequence of flops.
Therefore, $X^+_t$ is a minimal model of a very general fibre of $\mathcal X_{i_0} \longrightarrow B_{i_0}$.
Thus, by $( \ast )$, $X^+_t$ is isomorphic to a fibre of $\pi'$.
Since $X$ is smooth, $X^+_t$ is smooth, which implies 
$
[X]=[X^+_t] \in \mathcal S .
$
This finishes the proof of Theorem \ref{thm:deformations:2}.  
\end{proof}

\begin{remark} \label{rem:newpaper}  
Generalizing the above argument, we have shown in a joint work with Martinelli  \cite{MST16} that in arbitrary dimension, minimal models of general type and bounded volume form a bounded family. 
\end{remark}

\begin{remark} \label{rem:b2=2:unbounded}
The $\CP^1$-bundle $X_n:=\CP(\mathcal O_{\CP^2}(2n+1)\oplus \mathcal O_{\CP^2})$ over $\CP^2$ is a simply connected spin threefold with torsion free cohomology, $b_2(X_n)=2$ and $b_3(X_n)=0$.
However, $c_1^3(X_n)$ is unbounded for $n\to \infty$, and so the $X_n$'s do not belong to a finite number of deformation types.
Therefore, the general type assumption is needed in Theorem \ref{thm:deformations}.
The next example shows that the spin assumption is also necessary.
\end{remark}

\begin{example} \label{ex:deformations:gentype}
There are infinitely many deformation types of smooth projective threefolds of general type and with bounded Betti numbers.
Indeed, we start with a threefold $Y$ of general type which contains a smooth quadric surface $Q$.
Then, $Q$ contains smooth rational curves $C_i\subset Q$ such that the degree of the normal bundle of $C_i$ in $Q$ is unbounded.
The blow-up $X_i:=Bl_{C_i}Y$ satisfies
$
K_{X_i}^3=K_Y^3+2\deg(N_{C_i\slash X})+6 
$, 
see \cite[Proposition 4.8]{cascini-tasin} and \cite[Proposition 14]{OV95}.
Hence, $c_1^3(X_i)$ is unbounded in $i$, although the Betti numbers of $X_i$ do not depend on $i$.  
\end{example}

Using the classification of simply connected spin $6$-manifolds whose cohomology is not necessarily torsion free, see \cite{zhubr}, Theorem \ref{thm:deformations} implies easily the following; the details are analogous (but simpler) to the proof of Theorem \ref{thm:b_3=0} and so we leave them out.

\begin{corollary} \label{cor:6mfd-not-gentyp}
Let $X$ be any simply connected 
K\"ahler threefold with spin structure. 
Then there are infinitely many pairwise non-homeomorphic simply connected closed spin $6$-manifolds $M_i$, that have the same oriented homotopy type as $X$,
but which are not homeomorphic to any smooth complex projective variety of general type.  
\end{corollary}

Corollary \ref{cor:6mfd-not-gentyp} has the following amusing consequence: either there is a large supply of examples of homotopy equivalent K\"ahler manifolds with different Kodaira dimensions, or one can produce many more smooth manifolds which do not carry a K\"ahler structure although they have the homotopy type of a smooth complex projective variety of general type.
The known examples of homotopy equivalent K\"ahler manifolds with different Kodaira dimensions are rather scarce and come usually from particular complex surfaces, see for instance \cite{kollar}; it seems unlikely that they form a really big class.

\section{Mori fibre spaces}

In this section, we analyse the Chern numbers of  smooth Mori fibre spaces with spin structure  in dimension three.
This will be used in the proof of Theorem \ref{thm:w2=0} in Section \ref{sec:bounded}.

\subsection{Mori fibre spaces over curves} \label{sec:Mfs:curves}

In this subsection we prove the following result.

\begin{proposition} \label{prop:MFS:curves}
Let $(X_i)_{i\geq 0}$ be a sequence of compact K\"ahler manifolds of dimension $n$, admitting the structure of a Mori fibre space $X_i\longrightarrow C_i$ over a curve $C_i$. 
Suppose that 
\begin{enumerate}
\item the Betti numbers of the $X_i$'s are bounded;
\item there is an isomorphism $H^{2\ast}_{tf}(X_i,\Z)\stackrel{\sim} \longrightarrow H^{2\ast}_{tf}(X_0,\Z)$ between the even torsion free cohomology algebras, which respects the Pontryagin classes.
\end{enumerate}
Then the images of the Chern classes $c_1(X_i)$ and $c_2(X_i)$ in $H^{2\ast}_{tf}(X_0,\Z)$ are bounded.
In particular, for all $0\leq k\leq n/2$, the sequence of Chern numbers $c_1^{n-2k}c_2^k(X_i)$ is bounded.
\end{proposition}

\begin{proof}
Let $F_i$ denote the general fiber of $X_i\longrightarrow C_i$.
Then
$$
H^2(X_0,\Q)=x\cdot \Q\oplus y\cdot \Q,
$$
where $y=c_1(X_0)$, and $x\in H^2_{tf}(X_0,\Z)$ is primitive such that $[F_0]$ is a positive multiple of $x$.
The ring structure satisfies $x^2=0$, $y^{n-1}x=c_1^{n-1}(F_0)$ and $y^n=c_1^n(X_0)$.
From now on we will use the isomorphism $H^2(X_i,\Q)\cong H^2(X_0,\Q)$ to think about $x$ and $y$ as basis elements of $H^2(X_i,\Q)$.

We can write
$
c_1(X_i) = a_i\cdot x+ b_i\cdot y
$,
for some $a_i,b_i \in \Q$. 
Next, since $[F_i]^2=0$,  $[F_i]=\lambda_i \cdot x $, for some $\lambda_i\in \Z$.\footnote{In fact, $\lambda_i=\pm 1$, because $X_i\longrightarrow C_i$ has a rational section by Grabber--Harris--Starr's theorem, but we do not use this here.}
Moreover, since $F_i\subset X_i$ has trivial normal bundle,
$
c(X_i)|_{F_i}=c(F_i) 
$.
This implies
$$
c_1^{n-1}(F_i)=c_1(X_i)^{n-1}[F_i]= (a_i\cdot x+ b_i\cdot y)^{n-1} \lambda_i \cdot x=b_i^{n-1}\lambda_i \cdot y^{n-1}x.
$$
Since $F_i$ is Fano, $b_i$ and $\lambda_i$ are both nonzero and bounded, because Fano varieties of a fixed dimension form a bounded family. 

What follows is inspired by \cite[Theorem 4.2.3]{kollar}.
We have
$$
\chi(X_i,\mathcal O_{X_i})= \sum_{s\geq 0} \frac{1}{2^{n+2s}(n-2s)!}c_1(X_i)^{n-2s} A_s(p_1,...,p_s),
$$
where $A_s(p_1,...,p_s)$ is a polynomial in the Pontryagin classes.
Since $b_i$ is bounded and $x^2=0$, the above expression is a linear polynomial in $a_i$ with bounded coefficients,
$$
\chi(X_i,\mathcal O_{X_i})=\mu_1(i)\cdot a_i + \mu_0(i) .
$$
Here,
\begin{align*}
\mu_1 (i)&= \sum_{s\geq 0} \frac{1}{2^{n+2s}(n-2s)!}(n-2s)x\cdot (b_iy)^{n-1-2s} A_s(p_1,...,p_s)\\
&=\frac{1}{\lambda_i} [F_i] \cdot  \sum_{s\geq 0} \frac{1}{2^{n+2s}(n-1-2s)!} (b_iy)^{n-1-2s} A_s(p_1,...,p_s)\\
&=\frac{1}{2\lambda_i} [F_i] \cdot  \sum_{s\geq 0} \frac{1}{2^{n-1+2s}(n-1-2s)!} c_1(X_i)^{n-1-2s} A_s(p_1,...,p_s)\\
&=\frac{1}{2\lambda_i}\chi(F_i,\mathcal O_{F_i}) \\
&= \frac{1}{2\lambda_i} ,
\end{align*}
where we used that the Chern and Pontryagin classes of $X_i$ restrict to those of $F_i$.
The above computation shows $\mu_1(i)\neq 0$ for all $i$.

Since the Betti numbers of $X_i$ are bounded, and $X_i$ is K\"ahler, $\chi(X_i,\mathcal O_{X_i})$ attains only finitely many values.
Since $\mu_0(i)$ and $\mu_1(i)\neq 0$ are bounded, $a_i$ is bounded.  
This proves that $c_1(X_i)$ is bounded.
The boundedess of $c_2(X_i)$ follows from $p_1(X_i)=c^2_1(X_i)-2c_2(X_i)$. 
This proves the proposition.
\end{proof}

\begin{remark} \label{rem:MFS:curves}
The above proof shows that for $n=3$, one can replace condition (2) in Proposition \ref{prop:MFS:curves} by the following slightly weaker assumption: 
for each $i$ there is an isomorphism $H^2_{tf}(X_i,\Z)\longrightarrow H^2_{tf}(X_0,\Z)$ which respects the trilinear forms, given by cup products, and the linear forms, given by the first Pontryagin classes.
\end{remark}

\subsection{Unramified conic bundles over surfaces} \label{sec:Mfs:surface}
 
The aim of this subsection is to prove  

\begin{proposition} \label{prop:MFS:surfaces}
Let $(X_i)_{i\geq 0}$ be a sequence of smooth K\"ahler threefolds with the structure of an unramified conic bundle $f_i:X_i\longrightarrow S_i$  over a smooth K\"ahler surface $S_i$.
Suppose that
\begin{enumerate}
\item the Betti numbers of the $X_i$'s are bounded;
\item for each $i$, there is an isomorphism $H^2_{tf}(X_i,\Z)\cong H^2_{tf}(X_0,\Z)$ which respects the trilinear forms given by cup products. 
\end{enumerate} 
Then the sequence of Chern numbers $c_1^3(X_i)$ is bounded. 
\end{proposition}

We need the following lemma, which is well--known (at least) in the projective case.

\begin{lemma} \label{lem:Spin:Mfs} 
Let $X$ be a smooth K\"ahler threefold admitting the structure of an unramified conic bundle $f:X\longrightarrow S$ over a smooth K\"ahler surface $S$.
Then we have the following numerical equivalence on $S$:
$$
f_{\ast}K_X^2\equiv -4K_S .
$$ 
\end{lemma}

\begin{proof}
Note first that the N\'eron--Severi group $\NS(S)$ is generated by smooth curves $C\subset S$; this is clear if $S$ is projective and it follows easily from the classification of surfaces if $S$ is a non-projective K\"ahler surface, see \cite{barth-etal}. 
It thus suffices to compare the intersection numbers of $f_{\ast}K_X^2$ and $-4K_S$ with a smooth curve $C\subset S$.
This follows from an easy computation, where one uses that  $R:=f^{-1}(C)$  is a minimal ruled surface with normal bundle $f^\ast \mathcal O_S(C)|_R$, cf.\ \cite{mori-mukai} and \cite[Proposition 7.1.8]{iskovski}.
%

The following alternative (and slightly more general) argument was suggested to us by D.\ Kotschick.
We have $TX=f^\ast TS\oplus Tf$, where $Tf=\ker(f_\ast: TX\longrightarrow f^\ast TS)$ is the tangent bundle along the fibres of $f$. 
Hence,
$$
c_1^2(X)=f^\ast c_1^2(S)+2c_1(Tf)f^\ast c_1(S) +c_1^2(Tf) .
$$
As in Lemma \ref{lem:p1(Mfs)}, $Tf$ is stably isomorphic to $f^\ast E$ for a real rank three vector bundle $E$ on $S$. 
Since $Tf$ is a complex line bundle, we deduce $c_1^2(Tf)=p_1(Tf)\in f^\ast H^4(S,\Z)$ and so $f_\ast c_1^2(Tf)=0$.
Hence,
$$
f_\ast c_1^2(X)=
f_\ast(2c_1(Tf)f^\ast c_1(S))=4c_1(S) ,
$$ 
because $c_1(Tf)$ restricts to $c_1(\CP^1)=2$ on each fibre. 
\end{proof} 

\begin{proof}[Proof of Proposition \ref{prop:MFS:surfaces}]
Using the isomorphism $H^2_{tf}(X_i,\Z)\cong H^2_{tf}(X_0,\Z)$ which respects the trilinear forms given by cup products, we identify degree two cohomology classes of $X_i$ with those of $X_0$.
Using Poincar\'e duality, we further identify classes of $H^4_{tf}(X_i,\Z)$ with linear forms on $H^2_{tf}(X_i,\Z)\cong H^2_{tf}(X_0,\Z)$. 

The codimension one linear subspace ${f_i}^*\PP(H^2(S_i,\C))$ of $\CP(H^2(X_0,\C))$ is contained in the cubic hypersurface $\{\alpha\mid \alpha^3=0\}$. 
Passing to a suitable subsequence we can therefore assume that 
$$
f_i^*H^2(S_i,\C)\subset  H^2(X_0,\C)
$$ 
does not depend on $i$. 
Let $\ell_i \in H^4_{tf}(X_0,\Z)$ be the class of a fiber of $f_i$.
The action of this class on $H^2(X_0,\Q)$ has kernel $f_i^*H^2(S_i,\Q)$, and so $\ell_i\cdot \Q$ is independent of $i$.
Since $\ell_i$ is an integral class with $K_{X_i}.\ell_i=-2$, we may after possibly passing to another subsequence thus assume that $\ell_i=\ell$ does not depend on $i$.

For any class $y\in H^2(X_0,\Q)$ which does not lie in ${f_i}^*H^2(S_i,\Q)$, we have
$$
H^2(X_0, \Q)={f_i}^*H^2(S_i,\Q) \oplus y\cdot \Q \ \ \text{and}\ \ H^4(X_0, \Q)={f_i}^*H^2(S_i,\Q)\cdot y \oplus \ell\cdot \Q .
$$
In particular, $y^2=u y + \lambda \ell$ for some $\lambda \in \Q$ and $u \in {f_i}^*H^2(S_i,\Q)$.
Replacing $y$ by a suitable multiple of $y-\frac{1}{2}u$, we may thus assume that 
$$
y.\ell=-2\ \ \text{and}\ \ y^2 \in {f_i}^*H^4(S_i,\Q)=\ell\cdot \Q .
$$
Using this class, we have
$$
K_{X_i}=y+f_i^* z_i ,
$$ 
for some $z_i \in H^2(S_i,\Q)$. 
We claim that $z_i^2\in H^4(S_i,\Q)\cong \Q$ is a bounded sequence, which implies the theorem, because
$$
K_{X_i}^3=y^3+3f_i^* z_i^2\cdot y=y^3-6 z_i^2 ,
$$
and $y^3$ does not depend on $i$.

By Lemma \ref{lem:Spin:Mfs}, we have a numerical equivalence  
$$
{f_i}_*K_{X_i}^2 \equiv -4z_i \equiv -4K_{S_i}. 
$$
In particular, $z_i^2=K_{S_i}^2$ is bounded by the Betti numbers of $X_i$, as we want. 
This finishes the proof of Proposition \ref{prop:MFS:surfaces}.
\end{proof}

\subsection{Chern numbers of Mori fibre spaces with divisible canonical class}

The results of the previous subsections imply the following result.

\begin{corollary} \label{cor:Mfs:Kaehler}
Let $X$ be a smooth K\"ahler threefold which admits the structure of a Mori fibre space $f:X\longrightarrow B$.
If $c_1(X)$ is numerically divisible, then the Chern numbers of $X$ are determined up to finite ambiguity by the following invariants:
\begin{enumerate}
\item the Betti numbers of $X$;
\item the triple $(H^{2}_{tf}(X,\Z),F_X,p_1(X))$, where $F_X$ denotes the cubic form on  $H^{2}_{tf}(X,\Z)$, given by cup product, and $p_1(X)$ denotes the linear form given by the first Pontryagin class. 
\end{enumerate}  
\end{corollary}

\begin{proof}
The case where $B$ is a point follows from the boundedness of Fano threefolds \cite{iskovski}.
Next, note that  $F_X$ determines the trilinear form on $H^{2}_{tf}(X,\Z)$, given by cup product. 
If $B$ is a curve, then the assertion follows therefore from Proposition \ref{prop:MFS:curves} and Remark \ref{rem:MFS:curves}, and if $B$ is a surface, we conclude via Theorems  \ref{thm:MMP} and \ref{thm:p>2}, and Proposition \ref{prop:MFS:surfaces}.   
\end{proof}

\section{Proof of Theorem \ref{thm:w2=0}} \label{sec:bounded}
 
In the proof of Theorem \ref{thm:w2=0}, we will use the following result from \cite{cascini-tasin}.

\begin{lemma}[\cite{cascini-tasin}] \label{lem:blowup-point}
Let $Y$ be a smooth complex projective threefold and let $f:X\longrightarrow Y$ be the blow-up of a point of $Y$.
If $E\subset X$ denotes the exceptional divisor of $f$, then
\begin{enumerate}
\item the class $[E]\in H^2_{tf}(X,\Z)$ is determined up to finite ambiguity by the cubic form on $H^2_{tf}(X,\Z)$.
\item the class $[E]\in H^2_{tf}(X,\Z)$ determines uniquely the subspace $f^{\ast}H^2_{tf}(Y,\Z)\subset H^2_{tf}(X,\Z)$.
\end{enumerate}
\end{lemma}
\begin{proof}
The first assertion follows from \cite[Proposition 13]{OV95} and \cite[Proposition 3.3]{cascini-tasin}.
The second assertion follows from $H^2_{tf}(X,\Z)= f^{\ast}H^2_{tf}(Y,\Z) \oplus [E]\cdot \Z$, which shows that the cup product map
$$
\cup[E]:H^2_{tf}(X,\Z)\longrightarrow H^4_{tf}(X,\Z)
$$
has kernel  $f^{\ast}H^2_{tf}(Y,\Z)$.
\end{proof}
 
The following theorem implies Theorem \ref{thm:w2=0} from the introduction.

\begin{theorem} \label{thm:w2=0:2}
Let $X$ be a smooth K\"ahler threefold such that $c_1(X)$ is numerically divisible.
Then the Chern numbers of $X$ are determined up to finite ambiguity by the following invariants:
\begin{enumerate}
\item the Betti numbers of $X$; \label{item:thmw2=0:1}
\item the cubic form $F_X\in S^3H^{2}_{tf}(X,\Z)^{\vee}$, given by cup product;  
\item the linear form, given by the first Pontryagin class $p_1(X)\in H^2_{tf}(X,\Z)^{\vee}$. \label{item:thmw2=0:3}
\end{enumerate} 
\end{theorem}

\begin{proof} 
Let $X$ be a smooth complex projective threefold such that $c_1(X)$ is numerically divisible.
By Theorems \ref{thm:MMP} and \ref{thm:p>2}, there is a sequence of blow-downs to points in smooth loci 
$$
X=Y_r\longrightarrow Y_{r-1}\longrightarrow \dots \longrightarrow Y_1\longrightarrow Y_0=Y,
$$ 
such that $Y$ is either a smooth minimal model or a smooth Mori fiber space.
Moreover, $c_1(Y)$ is numerically divisible and the number of such blow-downs is bounded by $r\leq \rho(X)-1$, where $\rho(X)$ denotes the Picard number of $X$.

Let $E_i\subset Y_{i}$ be the exceptional divisor of $Y_{i}\longrightarrow Y_{i-1}$ and let $f_i:X\longrightarrow Y_i$ denote the natural map. 
By Lemma \ref{lem:blowup-point}, the class $[E_{r}]\in H^2_{tf}(X,\Z)$ is determined up to finite ambiguity by the cubic form on $H^2_{tf}(X,\Z)$, and it determines the subspace $f_{r-1}^{\ast}H^2_{tf}(Y_{r-1},\Z)\subset H^2_{tf}(X,\Z)$ uniquely.
Repeating this argument $r$ times, we conclude that for all $i$, the classes $f_i^{\ast}[E_i]$ as well as the subspaces $f_i^{\ast}H^2_{tf}(Y_i,\Z)\subset H^2_{tf}(X,\Z)$ are up to finite ambiguity determined by the cubic form on $H^2_{tf}(X,\Z)$. 

In particular, the isomorphism type of $H^2_{tf}(Y,\Z)$ together with the cubic form $F_Y$ given by cup product is up to finite ambiguity determined by the cubic form on $H^2_{tf}(X,\Z)$.
We aim to show that $p_1(Y)$ is determined up to finite ambiguity as well. 
In order to see this, we note
$$
H^2_{tf}(X,\Z)=f^{\ast}H^2_{tf}(Y,\Z) \oplus \bigoplus_{i=1}^r f_i^{\ast}[E_i]\cdot \Z .
$$ 
Moreover,
$$
c_1(X)=f^{\ast}c_1(Y)-2\sum_{i=1}^r  f_i^{\ast}[E_i] \ \ \text{and}\ \ c_2(X)=f^{\ast} c_2(Y) . 
$$
Hence,
\begin{align*}
p_1(X) &= f^{\ast} p_1(Y) +4 \sum_{i=1}^r f_i^{\ast}[E_i]^2 ,
\end{align*}
which proves the claim. 

We have thus proven that the triple $(H^2_{tf}(X,\Z),F_X,p_1(X))$ determines up to finite ambiguity the isomorphism type of $(H^2_{tf}(Y,\Z),F_Y,p_1(Y))$. 
Applying Proposition \ref{prop:gentype} and Corollary \ref{cor:Mfs:Kaehler}, we can therefore bound $c_1^3(Y)$ in terms of the Betti numbers of $X$ and the isomorphism type of $(H^2_{tf}(X,\Z),F_X,p_1(X))$.
Since $[E_i]^3=1$, 
\begin{align} \label{eq:c_1^3(X)}
c_1^3(X)=c_1^3(Y)-8r ,
\end{align} 
with $0\leq r \leq \rho(X)-1$.
Hence, $c_1^3(X)$ is also bounded in terms of the Betti numbers of $X$ and the isomorphism type  of $(H^2_{tf}(X,\Z),F_X,p_1(X))$.  
This proves the theorem.
\end{proof}

\subsection{Explicit bounds for $c_1^3$ of non-uniruled K\"ahler threefolds} 

Using suitable examples of $\CP^1$-bundles over surfaces, Kotschick showed that for uniruled spin threefolds $X$, the Betti numbers do not bound $c_1^3(X)$, see (proof of) \cite[Theorem 4]{kotschick-advances}.
In contrast to that result, the following corollary of the proof of Theorem \ref{thm:w2=0} shows that such a boundedness result is true if we restrict to non-uniruled spin threefolds.

\begin{corollary} \label{cor:ineq}
Let $X$ be a smooth K\"ahler threefold with spin structure. 
If $X$ is not uniruled, then 
$$
0\geq c_1^3(X)\geq \min \left(64\chi(X,\mathcal O_{X}) -8b_2(X)+8,-8b_2(X)+8\right) .
$$ 
\end{corollary}

\begin{proof}[Proof of Corollary \ref{cor:ineq}]
We use the notation from the proof of Theorem \ref{thm:w2=0}.
If $X$ is not of general type, then, by \cite[Corollary 1.4]{HP16}, the Kodaira dimension of $X$ satisfies $\operatorname{kod}(X)\in \{0,1,2\}$.
Hence, $Y$ is a minimal model which is not of general type and so $c_1^3(Y)=0$, see for instance \cite{CHP16}.
It then follows from (\ref{eq:c_1^3(X)}) that $c_1^3(X)=-8r$ for some $0\leq r \leq \rho(X)-1$.

Conversely, if $X$ is of general type, then $Y$ is a minimal model of general type and so the Miyaoka--Yau inequality yields $c_1^3(Y)\geq\frac{8}{3}c_1c_2(Y) $, see \cite{tsuji,zhang}.
Again by (\ref{eq:c_1^3(X)}), 
\begin{align*} 
c_1^3(X)=c_1^3(Y)-8r \geq \frac{8}{3}c_1c_2(Y) -8r .
\end{align*}
The corollary follows therefore from $c_1c_2(X)=c_1c_2(Y)$ and $0\leq r\leq b_2(X)-1$.
\end{proof} 

\begin{remark} \label{rem:thm:ineq:2}
Corollary \ref{cor:ineq} shows that non-uniruled K\"ahler threefolds $X$ with spin structure satisfy $K_X^3\geq 0$.
This fails without the spin assumption.
For instance, if $C$ and $E$ denote smooth projective curves of genus $g$ and $1$, then $X=Bl_{\Delta_C\times 0}(C\times C\times E)$ satisfies $K_X^3=-4g+10$, which is negative for $g\geq 3$. 
 \end{remark}

\section{Examples of unbounded Chern classes and deformation types} \label{sec:unbounded}

By Theorem \ref{thm:deformations}, the deformation types of K\"ahler structures of general type on a given spin $6$-manifold are bounded; 
by Theorem \ref{thm:w2=0}, the same holds for the Chern numbers of arbitrary K\"ahler structures. 
In this section, we show that both results are sharp.

\begin{proposition} \label{prop:dolgachev}
There is a simply connected spin $6$-manifold $M$, 
which  admits a sequence of K\"ahler structures $X_i$, such that $c_1(X_i)\in H^2(M,\Z)$ is unbounded. 
\end{proposition}

\begin{proof} 
Let $q\geq 3$ be an odd integer.
As in \cite{schreieder-tasin}, we consider the Dolgachev surface $S_q$, which is the elliptic surface, obtained from a general pencil of plane cubic curves $S\longrightarrow \CP^1$ by a logarithmic transformation of order $2$ and $q$ at two smooth fibres.
There is an h-cobordism $W_q$ between $S_q$ and $S_3$ which induces isomorphisms $H^2(S_q,\Z)\cong H^2(S_3,\Z)$.

We choose a cohomology class $\omega\in H^2(S_3,\Z)$  with $\omega^2\neq 0$ and such that the mod $2$ reduction of $c_1(S_3)+\omega$ vanishes.
Since $S_q$ and $S_3$ are h-cobordant, $w_2(S_q)\in H^2(S_3,\Z \slash 2\Z)$ does not depend on $q$ and so the mod $2$ reduction of $c_1(S_q)+\omega$ vanishes for all $q$. 
Since $h^{2,0}(S_q)=0$, there is a line bundle $L_q\in \Pic(S_q)$ with $c_1(L_q)=\omega$, see \cite[Section 2]{schreieder-tasin}.

Then the $\CP^1$-bundle $X_q:=\CP(L\oplus \mathcal O_{S_q})$ over $S_q$ is spin for all $q$.
Moreover, since $S_q$ is h-cobordant to $S_3$ and since $L\oplus \mathcal O_{S_q}$ extends as a complex vector bundle over the corresponding h-cobordism, it follows from the h-cobordism theorem that $X_q$ is diffeomorphic to $X_3$ for all $q$, cf. \cite{kotschick-JTOP,kotschick-advances,schreieder-tasin}.

From now on we think of $X_q$ as a complex projective structure on a fixed spin $6$-manifold $M$.
If $c_1(X_{q})\in H^2(M,\Z)$ is bounded, then, by Proposition \ref{prop:almostcx}, the almost complex structures underlying the $X_q$'s belong to finitely many homotopy classes. 
Let us therefore assume that the almost complex structures which underly $X_{q_1}$ and $X_{q_2}$ are homotopic.
Then there is an isomorphism of cohomology algebras,
$$
\phi:H^\ast(X_{q_1},\Z)\longrightarrow H^\ast(X_{q_2},\Z) ,
$$
which respects the Chern classes.
We use the following isomorphisms  
$$
H^2(X_{q_1},\Z)\cong H^2(S_3,\Z)\oplus y \Z\ \ \text{and}\ \ H^2(X_{q_2},\Z)\cong H^2(S_3,\Z)\oplus y\Z ,
$$
to identify the restriction of $\phi$ to degree two cohomology classes with an endomorphism of $H^2(S_3,\Z)\oplus y\Z$.
Using this identification, the ring structure is determined by $y^2=-\omega\cdot y$.
Since $\{\alpha\in H^2(S_3,\Z)\mid \alpha^2=\omega \alpha\}$ is an irreducible quadric, one easily proves that the elements of zero cube in $H^2(X_{q_i},\C)$ are given by the union of $H^2(S_3,\C)$ with an irreducible quadric.
In particular,
$
\phi(H^2(S_3,\Z))=H^2(S_3,\Z) $.
Since $\phi$ respects the ring structures, we deduce 
\begin{align} \label{eq:phi:H^4}
\phi(H^4(S_{q_1}, \Z))=H^4(S_{q_2},\Z) ,
\end{align}
where we identify $H^4(S_{q},\Z)$ via pullback with a subgroup of $H^4(X_q,\Z)$.
By assumptions, $\phi$ respects the second Chern class
$
c_2(X_{q_i})=c_2(S_{q_i})+ c_1(S_{q_i})(2y+\omega) 
$.
Since the Euler number of $S_q$ is independent of $q$, it follows from (\ref{eq:phi:H^4}) that
$$
\phi(c_1(S_{q_1})(2y+\omega))=c_1(S_{q_2})(2y+\omega).
$$
Multiplying this class with $y$ shows that it is non-zero.
Therefore, for $q_1$ fixed, $q_2$ is bounded, because $c_1(S_{q})$ is nonzero and divisible by $(q-2)$, see Proposition 3.7 in \cite[Section I.3]{friedman-morgan}. 
This proves the proposition. 
\end{proof}

\section*{Acknowledgments}
The first author is member of the SFB/TR 45.
During parts of this project, the second author was supported by the DFG Emmy Noether-Nachwuchsgruppe ``Gute Strukturen in der h\"oherdimensionalen birationalen Geometrie'' and thereby also member of the SFB/TR 45. 
We are very grateful to the referee for carefully reading a previous version of the paper.
We thank D.\ Kotschick for detailed comments and P.\ Cascini, M.\ Land, E.\ Sernesi, R.\ Svaldi and B.\ Totaro for conversations.

\end{document}